\documentclass{amsart}
\usepackage[english]{babel}
\usepackage{amsmath,amsfonts,amssymb,mathrsfs,amscd,comment}

\newtheorem{theo}{Theorem} 
\newtheorem{tmm}[theo]{Theorem}

\def\k{\mathbf{k}}

\def\k{\Bbbk}

\DeclareMathOperator{\Cotor}{Cotor}
\DeclareMathOperator{\Tor}{Tor}
\DeclareMathOperator{\Ext}{Ext}
\DeclareMathOperator{\Hilb}{Hilb}
\DeclareMathOperator{\cat}{cat}

\def\leq{\leqslant}
\def\geq{\geqslant}

\newcommand{\zk}{\mathcal Z_K}

\begin{document}

\title{Minimally non-Golod face rings and Massey products}

\author{Ivan Limonchenko}
\address{National Research University Higher School of Economics, Russian Federation}
\email{ilimonchenko@hse.ru}

\author{Taras Panov}
\address{Lomonosov Moscow State University, Russian Federation;\newline
Institute for Information Transmission Problems, Russian Academy of Sciences, Moscow;\newline
National Research University Higher School of Economics, Russian Federation}
\email{tpanov@mech.math.msu.su}


\thanks{The study has been funded within the framework of the HSE University Basic Research Program.}


\maketitle

\begin{abstract}
We give a correct statement and a complete proof of the criterion obtained in~\cite{G-P-T-W} for the face ring $\Bbbk[K]$ of a simplicial complex $K$ to be Golod over a field $\Bbbk$. (The original argument depended on the main result of~\cite{BJ}, which was shown to be false in~\cite{Kat}.)  We also construct an example of a minimally non-Golod complex $K$ such that the cohomology of the corresponding moment-angle complex $\zk$ has 
trivial cup product and a non-trivial triple Massey product.
\end{abstract}

Let $K$ be a simplicial complex on the vertex set $[m]=\{1,2,\ldots,m\}$. The {\emph{face ring}} $\Bbbk[K]:=\Bbbk[v_1,\ldots,v_m]/(v_{i_1}\!\cdots v_{i_r}\,|\,\{i_1,\ldots,i_r\}\notin K)$ is called {\emph{Golod}} (over $\Bbbk$) if the product and all higher Massey products in the Koszul complex $(\Lambda[u_1,\ldots,u_m]\otimes\Bbbk[K],d)$ are trivial.
By~\cite{Go}, $\Bbbk[K]$ is a Golod ring if and only if the \emph{Serre inequality} relating the Hilbert series of $\Ext_{\Bbbk[K]}(\k,\k)$ and $\Tor_{\k[v_1,\ldots,v_m]}(\k,\k[K])$ turns into equality. If $\Bbbk[K]$ is not Golod, but $\Bbbk[K_{[m]\setminus\{i\}}]$ is Golod for any $i\in [m]$, then $\Bbbk[K]$ is called {\emph{minimally non-Golod}} (over 
$\Bbbk$).

Given a topological pair $(X,A)$, its {\emph{polyhedral product}} $(X,A)^K$ is defined as $\bigcup_{\sigma\in K}(X,A)^{\sigma}$, for $(X,A)^{\sigma}:=\prod_{i\in [m]}X_i$, where $X_i=X$, if $i\in\sigma$ and $X_i=A$, otherwise. Recall that $\zk:=(\mathbb D^2,\mathbb S^1)^K$ and $\mathit{DJ}(K):=(\mathbb{C}P^\infty,\ast)^K$. The Koszul complex 
$(\Lambda[u_1,\ldots,u_m]\otimes\Bbbk[K],d)$ is quasi-isomorphic to the cellular cochains of $\zk$ with appropriate diagonal approximation~\cite[Lemma~4.5.3]{TT}; in particular, $H^*(\zk;\k)\cong\Tor_{\k[v_1,\ldots,v_m]}(\k,\k[K])$.

\begin{tmm}\label{maingolod}
Let $\Bbbk$ be a field. The following conditions are equivalent.
\begin{itemize}
\item[(a)] $\Bbbk[K]$ is a Golod ring over $\Bbbk$;
\item[(b)] the cup product and all Massey products in $H^{+}(\zk;\Bbbk)$ are trivial;
\item[(c)] $H_{*}(\Omega\zk;\Bbbk)$ is a graded free associative algebra;
\item[(d)] the Hilbert series satisfy the identity $\Hilb(H_{*}(\Omega\zk;\Bbbk);t)=\frac{1}{1-\Hilb(\Sigma^{-1}\widetilde{H}^{*}(\zk;\Bbbk);t)}$.
\end{itemize}
\end{tmm}
\begin{proof} Equivalence (a) $\Leftrightarrow$ (b) follows from~\cite[Theorem 4.5.4]{TT}.


For (a) $\Leftrightarrow$ (d), the theorem of Golod~\cite{Go} asserts that $\k[K]$ is a Golod ring if and only if the following identity for the Hilbert series holds:
\begin{equation}\label{hilb1}
  \Hilb\bigl(\Ext_{\Bbbk[K]}(\k,\k);t\bigr)=\frac{(1+t)^m}
  {1-\sum_{i,j>0}\beta^{-i,2j}(\k[K])t^{-i+2j-1}},
\end{equation}
where $\beta^{-i,2j}(\k[K])=\dim\Tor^{-i,2j}_{\k[v_1,\ldots,v_m]}(\k,\k[K])$. By~\cite[Prop.~8.4.10]{TT}, there is an isomorphism of algebras 
$H_{*}(\Omega\mathit{DJ}(K);\Bbbk)\cong\Ext_{\Bbbk[K]}(\Bbbk,\Bbbk)$. The loop space decomposition 
$\Omega\mathit{DJ}(K)\simeq\Omega\zk\times\mathbb{T}^m$~\cite[(8.16)]{TT} implies 
$\Hilb(H_{*}(\Omega\mathit{DJ}(K);\Bbbk);t)=\Hilb(H_{*}(\Omega\zk;\Bbbk);t)\cdot (1+t)^m$. Also, $\Hilb(\Sigma^{-1}\widetilde{H}^{*}(\zk;\Bbbk);t)=\sum_{i,j>0}\beta^{-i,2j}(\Bbbk[K])t^{-i+2j-1}$ by~\cite[Theorem~4.5.4]{TT}. Substituting this in~\eqref{hilb1} yields the identity of~(d).


We prove (c) $\Rightarrow$ (d). Let $Q=H_{>0}(\Omega\zk;\Bbbk)/(H_{>0}(\Omega\zk;\Bbbk)\cdot H_{>0}(\Omega\zk;\Bbbk))$ be 
the space of indecomposables. By assumption, $H_{*}(\Omega\zk;\Bbbk)=T\langle Q\rangle$, where $T\langle Q\rangle$ is the free associative algebra on the graded $\Bbbk$-module $Q$. The Milnor--Moore (bar) spectral sequence has the $E_2$-term
$E_{2}^{b}=\Tor_{H_{*}(\Omega\zk;\Bbbk)}(\Bbbk,\Bbbk)$ and converges to $\Sigma^{-1}H_{*}(\zk;\Bbbk)$. By assumption, $E_{2}^{b}\cong\Tor_{H_{*}(T\langle Q\rangle)}(\Bbbk,\Bbbk)\cong\Bbbk\oplus Q$ (as $\Bbbk$-modules), so $\Hilb(\Sigma^{-1}\widetilde{H}_{*}(\zk;\Bbbk);t)=\Hilb(E_{\infty}^{b};t)-1\leq \Hilb(E_{2}^{b};t)-1=\Hilb(Q;t)$. In particular, 
$\Hilb(T\langle\Sigma^{-1}\widetilde{H}_{*}(\zk;\Bbbk)\rangle;t)\leq \Hilb(T\langle Q\rangle;t)=\Hilb(H_{*}(\Omega\zk;\Bbbk);t)$.
The Adams (cobar) spectral sequence has $E_{2}^{c}=\Cotor_{H_{*}(\zk;\Bbbk)}(\Bbbk;\Bbbk)$ and converges to $H_{*}(\Omega\zk;\Bbbk)$. We have: $\Hilb(H_{*}(\Omega\zk;\Bbbk);t)=\Hilb(E_{\infty}^{c};t)\leq \Hilb(E_{2}^{c};t)\leq \Hilb(T\langle\Sigma^{-1}\widetilde{H}_{*}(\zk;\Bbbk)\rangle;t)$,
where the last inequality follows from the cobar construction (it turns to equality when all differentials in the cobar construction on $H_{*}(\zk;\Bbbk)$ vanish). 
By combining the two inequalities we obtain 
\[
  \Hilb\bigl(H_{*}(\Omega\zk;\Bbbk);t\bigr)=
  \Hilb\bigl(T\langle\Sigma^{-1}\widetilde{H}_{*}(\zk;  \Bbbk)\rangle;t\bigr)=
  \frac{1}{1-\Hilb(\Sigma^{-1}\widetilde{H}^{*}(\zk;\Bbbk);t)},
\]
proving (d).


To prove (d) $\Rightarrow$ (c) observe that the identity of (d) is equivalent to
$\Hilb(H_{*}(\Omega\zk;\Bbbk);t)=\Hilb(T\langle\Sigma^{-1}\widetilde{H}_{*}(\zk;\Bbbk)\rangle;t)$.
This implies that all differentials in the Adams cobar construction on $H_{*}(\zk;\Bbbk)$ are trivial. Thus, $H_{*}(\Omega\zk;\Bbbk)$ is a free associative algebra (on $\Sigma^{-1}\widetilde{H}_{*}(\zk;\Bbbk)$). 
\end{proof}

As a corollary of Theorem~\ref{maingolod} we get that $\Bbbk[K]$ is a Golod ring if and only if  both the Milnor--Moore and Adams spectral sequences for $\zk$ degenerate in the $E_2$-term. For the Golod complex $K$ with $\cat(\zk)>1$ constructed in~\cite{IK}, this gives a stronger lower bound for the order of vanishing differentials in the Milnor--Moore spectral sequence than the Ginsburg theorem~\cite{Gin}.  

In the case of flag $K$, it is proven in~\cite{G-P-T-W} that $\Bbbk[K]$ is Golod if and only if $\mathrm{cup}(\zk)=1$, and if $\k[K]$ is minimally non-Golod, then $\mathrm{cup}(\zk)=2$. In general, for a minimally non-Golod $\k[K]$, we have an upper bound $\mathrm{cup}(\zk)\leq 2$, which follows easily from Baskakov's description of the product in~$H^*(\zk;\k)$, see~\cite[Theorem 4.5.4]{TT}. Building upon a construction of~\cite{Kat}, we give an example of a minimally non-Golod complex $\mathcal K$ such that $\mathrm{cup}(\mathcal Z_\mathcal K)=1$ and $H^*(\mathcal Z_\mathcal K)$ has a non-trivial triple Massey product.

\begin{tmm} Let $\mathcal K$ be given by its minimal non-faces: $(1,2,3)$, $(4,5,6)$, $(7,8,9)$, $(1,4,7)$, $(1,2,4,5)$, $(5,6,7,8)$, $(2,3,7,8)$, $(2,3,5,6,7)$, $(1,2,4,6,8,9)$,
$(1,3,4,5,8,9)$, $(1,3,5,6,7,9)$, $(2,3,4,5,7,9)$, $(2,3,4,5,8,9)$, $(2,3,4,6,7,9)$, $(2,3,5,6,8,9)$. Then $\mathcal K$ is a $4$-dimensional minimally non-Golod complex such that $\mathrm{cup}(\mathcal Z_{\mathcal K})=1$ and there exists a non-trivial indecomposable triple Massey product of $5$-dimensional Koszul cohomology classes in $H^*(\mathcal Z_\mathcal K)\colon$
$$
\langle[v_1v_2u_3],[v_5v_6u_4],[v_7v_8u_9]\rangle=\{[v_1v_2v_5v_7v_8u_3u_4u_6u_9]\}.
$$ 
\end{tmm}

\begin{proof}
The description of the cup product for $\mathcal Z_{\mathcal K}$ \cite[Theorem 4.5.4]{TT} implies that $\mathrm{cup}(\mathcal Z_\mathcal K)=1$. The full subcomplex $\mathcal K_{[m]\setminus\{i\}}$ is Golod for any $i\in [m]$, by~\cite[Theorem~6.3~(5)]{Kat}. It suffices to show that the triple Massey product above is defined, non-trivial and indecomposable, since then $\mathcal K$ is not Golod itself and therefore is a minimally non-Golod complex. 

The triple Massey product above is defined and single-valued due to~\cite[Lemma~3.3]{L2019}, since $\widetilde{H}^*(\mathcal K_{\{1,2,3,4,5,6\}})\cong\widetilde{H}^*(\mathcal K_{\{4,5,6,7,8,9\}})=0$. It is non-trivial, since $[v_1v_2v_5v_7v_8u_3u_4u_6u_9]$ corresponds to a non-zero class in $H^4(\mathcal K)$ and $\dim\mathcal K=4$. 

It is indecomposable, since $\widetilde{H}^*(\mathcal K_{[m]\setminus\{1,2,3\}})\cong\widetilde{H}^*(\mathcal K_{[m]\setminus\{4,5,6\}})\cong\widetilde{H}^*(\mathcal K_{[m]\setminus\{7,8,9\}})=0$ and $\widetilde{H}^p(\mathcal K_{[m]\setminus\{1,4,7\}})\cong\Bbbk$ if $p=4$ and is zero otherwise, whereas $\widetilde{H}^q(\mathcal K_{\{1,4,7\}})\cong\Bbbk$ if $q=1$ and is zero otherwise. 
\end{proof}

It follows directly from~\cite[Theorem 6.3~(5),(7)]{Kat} that if $K$ is a minimally non-Golod simplicial complex such that $\mathrm{cup}(\zk)=1$ and there exists a non-trivial Massey product in $H^{+}(\zk)$, then $\dim(K)\geq\dim(\mathcal K)=4$ and $f_0(K)\geq f_0(\mathcal K)=9$.

We are grateful to Victor Buchstaber for fruitful discussions. The first author is a Young Russian Mathematics award winner and would like to thank its sponsors and jury.

%
%
%
%
\end{document}